\documentclass[a4paper]{article}
\usepackage{appendix}
\usepackage{amsmath, amsthm, amssymb}
\usepackage{mathtools}
\usepackage[backref=page, colorlinks=true, linkcolor=blue, urlcolor=blue, citecolor=blue]{hyperref}
\usepackage{faktor}

\renewcommand*{\backref}[1]{}
\renewcommand*{\backrefalt}[4]{%
    \ifcase #1 Not cited.%
    \or        Cited on page~#2.%
    \else      Cited on pages~#2.%
    \fi}

\let\phi\varphi 
\let\epsilon\varepsilon 

\newtheorem{theorem}{Theorem}[section]

\newtheorem{remark}[theorem]{Remark}
\newtheorem{corollary}[theorem]{Corollary}
\newtheorem{definition}[theorem]{Definition}
\newtheorem{lemma}[theorem]{Lemma}
\newtheorem{proposition}[theorem]{Proposition}

\newcommand{\RR}{\mathbb{R}}
\newcommand{\ZZ}{\mathbb{Z}}
\newcommand{\rp}{\mathbb{R}\mathrm{P}^2}

\newcommand{\ncg}{\mathcal{NC}}
\newcommand{\fncg}{\mathcal{NC}^\dagger}

\title{Non-orientable surfaces have stably unbounded homeomorphism group}
\author{Lukas Böke\footnote{boeke@math.lmu.de}}
\date{}

\begin{document}

\maketitle

\begin{abstract}
Building on recent work of Bowden, Hensel and Webb, we prove that the groups of homeomorphisms of the real projective plane and Möbius strip which are isotopic to the identity have an infinite dimensional space of non-trivial homogeneous quasi-morphisms.
In particular, we show that they are stably unbounded, completing the answer to a question posed by Burago, Ivanov and Polterovich on the boundedness of diffeomorphism groups of surfaces.
\end{abstract}

\section{Introduction}
In this article, we study the (un-)boundedness of norms on the group of homeomorphisms and diffeomorphisms of the Möbius strip and real projective plane that are isotopic to the identity.
This type of question first appeared in the work of Burago, Ivanov and Polterovich \cite{BIP}, where the authors considered the identity components of groups of compactly supported $C^\infty$-diffeomorphisms of various manifolds, and found that for all compact 3-manifolds and certain compact 4-manifolds these groups only admit conjugation-invariant norms that are bounded.
Their results have been extended to higher dimensional manifolds in \cite{tsuboi}.
For surfaces, the situation is different:
while the sphere again only admits conjugation-invariant norms that are bounded, almost all other surfaces have been found to have unbounded norms.
The main result of this article answers this question for the last two remaining cases: 

\begin{theorem}\label{thm:intro}
The groups of homeomorphisms of the real projective plane $\rp$ and Möbius strip $M$ which are isotopic to the identity have infinite-dimensional spaces of non-trivial homogeneous quasi-morphisms.
In particular, they are stably unbounded with respect to commutator length in the sense of \cite{BIP}.
\end{theorem}

Therefore, all compact surfaces except the sphere admit unbounded norms.
Further, we give a condition which implies positive scl of a homeomorphism of the (puctured) real projective plane or Möbius strip, in analogy to section 5 of \cite{BHW2}.

To prove unboundedness, the central tool is Bavard duality \cite{bavard}, which links a norm, namely commutator length, to the existence of non-trivial homogeneous quasi-morphisms.
An explicit construction of quasi-morphism was given by Brooks in \cite{Brooks} for free groups.
This construction was generalised in \cite{Bestvina} by Bestvina and Fujiwara to groups of isometries of a hyperbolic space, which they applied to (subgroups of) mapping class groups of hyperbolic surfaces.
They extended their results to non-orientable surfaces in \cite{Bestvina-refinement}.

In their papers, they consider the action of the mapping class group on the curve graph (introduced by Harvey \cite{Harvey}, shown to be hyperbolic by Masur and Minsky \cite{masur-minsky}), which satisfies a newly identified property, namely weak proper discontinuity (see section \ref{subsec:actions-on-cg}).

An analogous strategy was developed by Bowden, Hensel and Webb in \cite{BHW1}, where the action of the identity component of the homeomorphism group of an orientable surface with positive genus on the newly introduced \emph{fine} curve graph was studied.

This strategy was adapted to non-orientable surfaces of genus at least three in \cite{kimura-kuno-qms}.
In \cite{BHW2}, the action of these groups on the fine curve graph is studied further, and a tool is provided that characterises some of the elements whose powers have increasingly large commutator length.
In addition, \cite{BHW2} provides a connection between certain foliations on surfaces and boundary points of the corresponding fine curve graph.
In this article, we aim to explain that the techniques and results of Bowden, Hensel and Webb generalise to non-orientable surfaces with genus 1.

\paragraph{Organisation of the paper.}
In the next section, we provide some background on the papers of Bestvina and Fujiwara \cite{Bestvina} \cite{Bestvina-refinement} and Bowden, Hensel and Webb \cite{BHW1} \cite{BHW2}.
In the third section, we combine results of \cite{kuno-hyperbolicity} and \cite{BHW1} to show that the Möbius strip and real projective plane have hyperbolic and unbounded fine curve graphs.
In the fourth section, we use that result and the methods of \cite{BHW2} to prove the main results of this article.

\paragraph{Acknowledgements.}
The author wants to thank Sebastian Hensel and Javier de la Nuez Gonzalez.

\section{Background}
In this section, we provide the background needed for the following sections.
We start with introductions to the objects we are interested in.
After fixing some notations and terminology for groups of homeomorphisms of surfaces, we provide some background on quasi-morphisms on groups and curve graphs of surfaces.
The next two sections describe the geometry of the curve graphs and the actions of the groups we discussed on curve graphs.
Section \ref{subsec:bsf} provides some less-known tools tailored to the strategy of \cite{BHW2}, which we follow closely.

\subsection{Groups of homeomorphisms}
\paragraph{Definitions.}
Let $M$ be a compact manifold, possibly with boundary, possibly non-orientable.
We denote by $\operatorname{Homeo}(M)$ the groups of homeomorphisms which fix the boundary pointwise, and use $\operatorname{Homeo}_0(M)$ for the component containing the identity.
For a finite set $P\subset M$, we denote by $\operatorname{Homeo}(M\setminus P)$ the subgroup of $\operatorname{Homeo}(M)$ which fixes \emph{the set} $P$.

The \emph{mapping class group} of $M$ is defined as 
\[\operatorname{Mcg}(M) \coloneqq \faktor{\operatorname{Homeo}(M)}{\operatorname{Homeo}_0(M)},\]
and analogously for $\operatorname{Mcg}(M\setminus P)$.
Note that in most of the literature on mapping class groups, the homeomorphism groups in the above definition are replaced with the respective groups of orientation preserving homeomorphisms.
As we are mostly concerned with non-orientable surfaces, we have to drop this condition.

For the respective subgroups of $C^\infty$-diffeomorphisms, we will use the notations $\operatorname{Diff}$ and $\operatorname{Diff}_0$.

\paragraph{Point-pushing and surface braids.}
Assume now that $M$ is two-dimensional.
A way to understand at least some elements of $\operatorname{Homeo}_0(M)$ is through point-pushes, which arise from the Birman exact sequence (see Theorem 4.6 in \cite{primer}):
\begin{equation}\label{eq:birman-ses} 1 \rightarrow \pi_1(M,p) \xrightarrow{\textsc{Push}} \operatorname{Mcg}(M\setminus \{p\}) \xrightarrow{\textsc{Forget}} \operatorname{Mcg}(M) \rightarrow 1 .\end{equation}
Formally, one can see a path in $\pi_1(M,p)$ as a homotopy of maps from a singleton to $M$, and define its image under $\textsc{Push}$ as the homotopy class of extensions of such a homotopy to the whole surface.
We will refer to classes in the image of \textsc{Push} as \emph{point-pushes}.

One can also push more than one point:
a more general version of the Birman exact sequence (see Theorem 9.1 in \cite{primer}) accomplishes this by replacing the fundamental group of the surface in \eqref{eq:birman-ses} with the fundamental group of the unordered configuration space of the surface.
In this case, classes in the image of \textsc{Push} are called \emph{multi-point-pushes} or \emph{surface braids} in analogy to the Artin braid groups.
We can continue the analogy by considering ordered configuration spaces instead of unordered ones, and this way we get the \emph{pure surface braids}.

\paragraph{Pseudo-Anosov maps of non-orientable surfaces.}
The Nielsen-Thurston classification (see e.g. chapter 13 of \cite{primer}) tells us that for orientable compact manifolds, each mapping class can be periodic, reducible or pseudo-Anosov, and that pseudo-Anosov mapping classes are not periodic and not reducible.
In particular, we know that a pseudo-Anosov mapping class contains a representative which is a pseudo-Anosov homeomorphism, i.e. there is a number $\lambda > 0$ and a pair of transverse measured foliations (which we will refer to as the vertical and horizontal foliations) which are preserved and whose transverse measures are multiplied by $\lambda$ respectively $\lambda^{-1}$.
We call such a homeomorphism a \emph{Thurston representative} of its mapping class.
Note that the definition of a pseudo-Anosov homeomorphism does not depend on the orientability of the surface.
In \cite{wu}, a completely analogous theorem to the Nielsen-Thurston classification is proven for compact non-orientable surfaces.

\subsection{Quasi-morphisms and commutator length}
\paragraph{Definitions.}
For any group $G$, one can consider maps $\phi\colon G\rightarrow \RR$ with the following property:
for two elements $g$, $h$ of $G$, the map $\phi$ behaves like a homomorphism, but with a uniformly bounded error, i.e. $\lvert \phi(gh) - \phi(g) - \phi(h)\rvert$ is bounded by some $D>0$ which is independent of $g$ and $h$.
We then call $\phi$ a \emph{quasi-morphism}, and the smallest possible bound $D$ is called the \emph{defect} of $\phi$.
We say that a quasi-morphism is trivial if it is a homomorphism of groups, i.e. the defect can be chosen to be 0.
Further, we call quasi-morphisms \emph{homogeneous} if they restrict to homomorphisms on cyclic subgroups, i.e. if for any $k\in\ZZ$ and $g\in G$ we have $\phi(g^k) = k\cdot \phi(g)$.

Following \cite{Bestvina}, we define the following spaces of quasi-morphisms:
first, let $\mathcal{V}(G)$ be the vector-space of all quasi-morphisms $G\to \RR$.
We denote by $\operatorname{BDD}(G)\subset \mathcal{V}(G)$ the subspace of bounded functions and by $\operatorname{Hom}(G,\RR)$ the space of group homomorphisms $G\to \RR$.
We also set
\[ \operatorname{QH}(G) = \faktor{\mathcal{V}(G)}{\operatorname{BDD}(G)}\]
and
\[ \widetilde{\operatorname{QH}}(G) = \faktor{\operatorname{QH}(G)}{\operatorname{Hom}(G,\RR)}.\]
Note that $\widetilde{\operatorname{QH}}(G)$ can also be identified with the kernel of the comparison map between bounded and unbounded group cohomology in degree 2.
For more details on bounded cohomology, we refer to \cite{frigerio}.

\paragraph{Bavard duality.}
It is known that the groups $\operatorname{Homeo}_0(M)$ of compact manifolds are perfect \cite{anderson}, i.e. any element of the group can be written as a product of commutators, and the same is true for $\operatorname{Diff}_0$, see e.g. \cite{mann}.
The minimal number of commutators that appear in such a factorisation of a homeomorphism is called \emph{commutator length}, usually denoted by cl, and it is an important example of a conjugation-invariant norm on the group $\operatorname{Homeo}(M)$ as defined in \cite{BIP}.
Such norms have a so-called \emph{stabilisation}, which takes the following form for commutator length:
\[ \operatorname{scl}(g) \coloneqq \lim_{n\rightarrow \infty} \frac{\operatorname{cl}(g^n)}{n} \]
for $g\in \operatorname{Homeo}(M)$.
We call a group \emph{stably unbounded} if the stabilisation of some norm takes positive values.

An important connection between quasi-morphisms and scl is provided by the following theorem:
\begin{theorem}[Bavard duality, \cite{bavard}]\label{thm:bavard-duality}
    For all $g\in \operatorname{Homeo}(M)$, we have
    \[\operatorname{scl}(g) = \sup_\phi \frac{\lvert \phi(g)\rvert}{2D(\phi)}\]
    where we take the supremum over all non-trivial homogeneous quasi-morphisms.
\end{theorem}

\subsection{Curve graphs and variations}
\paragraph{Definitions.}
Curve graphs and curve complexes were defined by Harvey in \cite{Harvey}.
These are usually defined in terms of essential curves, but for our purposes it will be useful to only work with non-separating curves.
The main definition here is:
the \emph{non-separating curve graph} $\ncg(M)$ of a 2-manifold (possibly with boundary) $M$ is the graph with: \begin{itemize}
    \item vertices corresponding to the homotopy classes of non-separating closed curves in $M$
    \item an edge between two vertices, if the corresponding homotopy classes have disjoint representatives
\end{itemize}
To study homeomorphisms that are isotopic to the identity, we have to get rid of the homotopy classes:
the \emph{fine non-separating curve graph} accomplishes this.
It was introduced in \cite{BHW1} and is defined by:\begin{itemize}
    \item vertices which each represent a non-separating closed curve in $M$
    \item an edge between two vertices, if the corresponding curves are disjoint
\end{itemize}

In the case of manifolds with (non-orientable) genus at most 2, we have to modify these definitions slightly and allow for up to one intersection point between two representatives or curves respectively.
We also remark that in \cite{kuno-hyperbolicity}, the main theorem is proved for a variation of the definition allowing for up to two intersection points, which is shown to be quasi-isometric to the curve graphs we defined.

We also define a coarse notion of convexity:
a subset $Y\subseteq X$ of a geodesic metric space is \emph{$K$-quasi-convex} for $K>0$, if for any two points in $Y$ any geodesic between them lies in a $K$-neighbourhood of $Y$.

\paragraph{Distance bounds.}
As the condition for an edge between two vertices is defined in terms of intersection points of curves, it is natural to consider the \emph{geometric intersection number} of two curves $\alpha$, $\beta$, which we define as
    \[ i(\alpha,\beta) \coloneqq \lvert \alpha \cup \beta\rvert.\]
For the corresponding homotopy classes, we take the minimum over all possible representatives.
For details cf. section 1.2.3 in \cite{primer}.

These intersection numbers can be useful to find upper bounds on distances in curve graphs, one example is the following lemma due to \cite{masur-minsky}:
\begin{lemma}[cf. Lemma 2.1 in \cite{masur-minsky} and Lemma 4.3 in \cite{kuno-hyperbolicity}]\label{lem:upper-bound-intersection-no}
    For any two non-separating curves $\alpha$, $\beta$,
    \[d_{\ncg}(\alpha,\beta) \leq 2i(\alpha,\beta) +1 .\]
\end{lemma}

Lower bounds are harder to find, but the \emph{covering criterion} of \cite{BHMMW} provides one, following the strategy of Hempel \cite{hempel}.
This is phrased in terms of \emph{elevations} instead of lifts, which are connected components of the preimages of a curve under a (branched) covering map.
For our purposes, we need to modify their statement in the following way:

\begin{lemma}[cf. Lemma 6.4 of \cite{BHMMW}]\label{lem:covering-criterion}
    Let $K\geq 0$.
    There exists a finite-sheeted branched covering $M \rightarrow N$ (depending only on $K$) such that any two curves $\alpha$ and $\beta$ with $d_{\fncg(N)}(\alpha,\beta)\leq K$ which are disjoint from the branching locus have disjoint elevations to $M$.
\end{lemma}

The proof is analogous to the proof in \cite{BHMMW}.

\subsection{Hyperbolicity}
We begin by defining hyperbolicity through the Gromov product:

\begin{definition}
Let $(X,d)$ be a metric space, and $x,y,w\in X$.
The \emph{Gromov product} is defined as:
\begin{equation*}
    ( x\cdot y)_w \coloneqq \frac{1}{2}\big( d(w,x) + d(w,y) - d(x,y)  \big).
\end{equation*}
We say that $X$ is $\delta$-hyperbolic, if 
\[\forall w,x,y,z \in X\colon ( x\cdot z )_w \geq \min\left\{ ( x\cdot y)_w, (y\cdot z)_w \right\} - \delta.\]
\end{definition}

We can use this framework to describe a ``boundary at infinity" of $(X,d)$:
fix $w$ as the base point in $X$.
We call a sequence $(a_i)$ in $X$ \emph{admissible}, if for $i,j \rightarrow \infty$, we have
\[(a_i\cdot a_j)_w \rightarrow \infty.\]
and we call two admissible sequences $(a_i)$ and $(b_j)$ equivalent if for $i,j\rightarrow \infty$, we have
\[(a_i\cdot b_j)_w \rightarrow \infty.\]
We can then define the \emph{Gromov boundary} $\partial_\infty X$ of $X$ as the set of equivalence classes of admissible sequences in $X$.
One can extend the Gromov products to boundary points $\tau, \upsilon \in \partial_\infty X$:
\[(\tau\cdot \upsilon)_w \coloneqq \inf \left\{\left. \liminf_{i,j\rightarrow \infty}\, (a_i\cdot b_j)_w \,\right|\, (a_i)\in \tau,\, (b_j)\in \upsilon \right\}\]

We will also need a way of proving that a space is $\delta$-hyperbolic:
the following is a well-established technique, we cite Bowditch's version \cite{Bowditch_unif-hyp}.

\begin{theorem}[Guessing geodesics]\label{thm:bowditch-guessing-geodesics}
    Let $X$ be a graph and $D>0$ a number.
    Suppose that for each pair of distinct vertices $x,y$ of $X$ we have chosen a connected subgraph $P(x,y)$ containing $x$ and $y$.
    Suppose that if $d_X(x,y) = 1$, the diameter of $P(x,y)$ is at most $D$, and in addition we have for all pairwise distinct $x,y,z$:
    \[ P(x,z) \subseteq N_D(P(x,y)\cup P(y,z)). \]
    Then $X$ is $\delta(D)$-hyperbolic, and each $P(x,y)$ is $B(D)$-Hausdorff close to a geodesic joining $x$ to $y$.
\end{theorem}

\subsection{Actions on curve graphs}\label{subsec:actions-on-cg}
One can easily check that the groups $\operatorname{Homeo}_0$ and $\operatorname{Mcg}$ of a surface act on the fine curve graph respectively curve graph of that surface by isometries.

In \cite{Bestvina}, Bestvina and Fujiwara introduced the notion of weak proper discontinuity, which will be the key to finding elements with positive stable commutator length.
The article \cite{Bestvina} is a continuation of a row of works of Fujiwara and Epstein (\cite{epstein-fujiwara}, \cite{fujiwara-gromov-hyperbolic}, \cite{fujiwara-amalgamated}), which extend the work of Brooks \cite{Brooks} on free groups.

First, recall that for two hyperbolic isometries $f$ and $g$ of a $\delta$-hyperbolic space $(X,d)$, we can find invariant quasi-axes $\alpha$ and $\beta$ of quality $L$ respectively $L'$ by taking the orbit of some point, e.g. $\alpha = (a_k) \coloneqq f^k(x_0)$ for $k\in\ZZ$ and $x_0\in X$, and analogously for $\beta = (b_k)$.
We then say that $f$ and $g$ (respectively their axes) are \emph{quasi-equivalent}, if there exists a constant $B = B(L,L',\delta)$ such that:
for any $D>0$, there is an isometry $h$ that sends a segment of length $D$ in $\alpha$ into a $B$-neighbourhood of $\beta$.
If $f$ is quasi-equivalent to its inverse, then we say that $f$ is \emph{quasi-invertible}.

These quasi-geodesics define points in the Gromov boundary of $X$, and we can give an alternative definition of quasi-equivalence in terms of boundary points:
let $\tau_\pm \coloneqq [a_{\pm k}]$ and $\upsilon_\pm \coloneqq [b_{\pm k}]$ as equivalence classes of admissible sequences as in the previous section.
Then $f$ and $g$ (respectively their axes) are quasi-equivalent, if there exists a sequence $h_k$ of isometries of $X$ such that $\lim_{k\rightarrow \infty} h_k(\tau_\pm) = \upsilon_\pm$.

Further, we call two hyperbolic elements \emph{dependent}, if they have quasi-axes which contain rays that are in finite distance of each other, and \emph{independent} otherwise.
If $G$ contains two independent hyperbolic elements, we call the action \emph{non-elementary}.

The central definition in this section is:

\begin{definition}\label{def:wpd}
    Let $G$ be a group acting on the hyperbolic metric space $X$.
    This action is \emph{weakly properly discontinuous (WPD)}, if the following holds: \begin{itemize}
        \item $G$ is not virtually cyclic
        \item some $g\in G$ acts on $X$ as a hyperbolic isometry
        \item for all $g\in G$ acting hyperbolically and every point $x\in X$ and any constant $C>0$, we have an $N>0$ such that the set 
        \[\{\gamma \in G \mid d(x,\gamma(x)) \leq C, d(g^N(x),\gamma(g^N(x)))\leq C \}\]
        is finite.
    \end{itemize}
    \end{definition}
Note that when we restrict an action $G\curvearrowright X$ to a subgroup $H<G$, we only need to recheck the first two conditions to see that the induced action of $H$ on $X$ satisfies WPD.
The following proposition shows why WPD is a useful property:

\begin{proposition}[Proposition 6 of \cite{Bestvina}]\label{prop:useful-wpd}
    Suppose that $G$ and $X$ satisfy WPD.
    Then, the following all hold: 
    \begin{enumerate}
        \item for every hyperbolic $g\in G$, the centraliser $C(g)$ is virtually cyclic,
        \item for every hyperbolic $g\in G$ and every $(K,L)$-quasi-axis $l$ for $g$, there is a constant $M$ depending on $g$, $K$ and $L$ such that if two translates $l_1$ and $l_2$ of $l$ contain (oriented) segments of length $>M$ that are oriented $B(K,L,\delta)$-close, then $l_1$ and $l_2$ are oriented $B(K,L,\delta)$-close and moreover the corresponding conjugates $g_1$ and $g_2$ of $g$ have positive powers which are equal,
        \item the action of $G$ on $X$ is non-elementary,
        \item two elements of $G$ are quasi-equivalent if and only if they have positive powers which are conjugate,
        \item there exist two independent hyperbolic elements in $G$ which are not quasi-equivalent.
    \end{enumerate}
\end{proposition}

\begin{remark}
    The fifth statement of the previous proposition is slightly stronger than the statement in \cite{Bestvina}, but the proof given there allows this conclusion:
    in the proof, one works in a free subgroup (``Schottky subgroup") generated by independent hyperbolic elements, up to maybe replacing them with higher powers.
\end{remark}

The main result of \cite{Bestvina} is:
\begin{theorem}[Theorem 1 of \cite{Bestvina}]\label{thm:bestvina-fujiwara-main-thm}
Suppose a group $G$ acts on a $\delta$-hyperbolic graph $X$ by isometries.
Suppose also that the action is non-elementary and that there exist independent hyperbolic elements $g_1,g_2 \in G$ which are not quasi-equivalent.
Then, the space of non-trivial homogeneous quasi-morphisms $\widetilde{\operatorname{QH}}(G)$ is infinite-dimensional.
\end{theorem}

\paragraph{Extension to non-orientable surfaces}
In \cite{Bestvina-refinement}, Bestvina and Fujiwara extended their results to non-orientable surfaces.
We will later use their results in the case that $N$ is the real projective plane and $P\subset N$ a set of at least three punctures.
One can show:

\begin{proposition}[cf. Theorem 5.15 of \cite{Bestvina-refinement}]\label{prop:wpd-downstairs}
    The action of $\operatorname{Mcg}(N\setminus P)$ on $\ncg(N\setminus P)$ has WPD.
\end{proposition}

\begin{remark}
Note that we can now apply Proposition \ref{prop:useful-wpd} to the action of $\operatorname{Mcg}(N\setminus P)$ on $\ncg(N\setminus P)$.
In particular, we get that $\operatorname{Mcg}(N\setminus P)$ contains pseudo-Anosov classes which are not quasi-invertible.    
\end{remark}

In particular, we immediately get:
\begin{corollary}[cf. Theorem 1.1 of \cite{Bestvina-refinement}]
    The group $\operatorname{Mcg}(N\setminus P)$ is stably unbounded.
\end{corollary}

\subsection{BSF structures and resolving coverings}\label{subsec:bsf}
In this section, we introduce some geometric notions from \cite{BHW2} and slightly adapt them to the non-orientable setting.

First, a \emph{singular flat structure} on a (possibly non-orientable) surface $S$ will mean a singular flat metric with discrete cone points, whose cone angles are all multiples of $\pi$.
We will usually denote the set of 1-prongs (also called angle-$\pi$ cone points) by $P$.
We can obtain these from gluing polygons.
This leads to:
\begin{definition}
    A \emph{bifoliated singular flat (BSF) structure} is a tuple $q = (\mathcal{F}_v, \mathcal{F}_h,d)$ consisting of\begin{itemize}
        \item a singular flat metric $d$ as before
        \item transverse singular foliations $\mathcal{F}_v$ and $\mathcal{F}_h$ which are locally geodesic and orthogonal with respect to $d$ except at the cone points
    \end{itemize}

We call leaves of the BSF structure \emph{regular}, if they don't contain a singular point, and \emph{singular}, if they are a maximal leaf segment or ray bounded by 
singularities, but with no singularities in its interior.
A \emph{pseudo-leaf} is a connected subset of the surface which is a finite union of singular leaves.

Further, we will call a foliation $\mathcal{F}$ \emph{minimal}, if every regular leaf of $\mathcal{F}$ is dense.
If additionally every pseudo-leaf is simply connected and contains at most one 1-prong, then we call $\mathcal{F}$ \emph{ending}.
\end{definition}

As we have to allow for 1-prongs, geodesics in BSF structures are not very well-behaved.
To control this, one introduces a few new notions:
given a non-orientable surface $N$ with BSF structure $q$, we consider the orientable finite branched coverings $\Sigma \rightarrow N$ with branches of order at least two at the 1-prongs.
Such a covering will again have a BSF structure, and ending foliations will lift to ending foliations.
See section 2 of \cite{BHW2} for more details.

We introduce another notion: 
following Koberda \cite{Koberda}, we say a covering $\operatorname{pr}\colon X\rightarrow Y$ is \emph{characteristic} if the subgroup $\operatorname{pr}_*(\pi_1(X)) < \pi_1(Y)$ is characteristic, i.e. it is fixed by the automorphism group of $\pi_1(Y)$.
We will also call branched coverings \emph{characteristic} if they restrict to covering maps away from the branch set.
It is easy to see that such coverings exist, as any subgroup $H$ of finite index in a finitely generated group $G$ contains a characteristic subgroup $C\subseteq H$ with finite index in $G$.
Characteristic coverings have a very useful property:

\begin{lemma}[folklore]\label{lem:char-covers-have-all-lifts}
If $\operatorname{pr}\colon X\rightarrow Y$ is a characteristic (branched) covering with $X$ and $Y$ manifolds, then every homeomorphism of $Y$ lifts to a homeomorphism of $X$.
\end{lemma}

\begin{proof}
Let $X$ and $Y$ be manifolds, and $\operatorname{pr}\colon X\rightarrow Y$ a characteristic covering.
Let $f\colon Y\rightarrow Y$ a homeomorphism.
We apply the \emph{lifting criterion} (see e.g. Proposition 1.33 in \cite{Hatcher}) to $f\circ \operatorname{pr}$:
as $\operatorname{pr}(\pi_1(X)) < \pi_1(Y)$ is characteristic, we know that $f_*(\operatorname{pr}_*(\pi_1(X))) = \operatorname{pr}(\pi_1(X))$.
Therefore, $f\circ \operatorname{pr}$ satisfies the lifting criterion, and we get a lift $\hat{f}\colon X\rightarrow X$.
By the same argument, up to a Deck transformation, the inverse of $f$ lifts to the inverse of $\hat f$ as well, which shows that $\hat{f}$ is a homeomorphism.

In the case where $\operatorname{pr}$ is a branched covering and $f\colon Y\rightarrow Y$ a homeomorphism that preserves the branching points, we first apply the above to a restriction of $\operatorname{pr}$ and $f$ away from the branching points, and then extend the resulting lifts to all of $X$.
\end{proof}

For a non-orientable surface $N$ with BSF structure $q$, we call the orientable finite characteristic branched coverings $\Sigma \rightarrow N$ with branches of order at least two at the 1-prongs \emph{$P$-resolving}.
The specific choice of the $P$-resolving covering will not change the arguments we present in the later sections, we will however assume that we use the same covering throughout the paper.
We also note that now closed geodesics in the $P$-resolving covering can only intersect in geodesic segments, but not bound bigons (see Lemma 2.7 of \cite{BHW2}).

Also note that the universal covering $\Tilde{\Sigma}\rightarrow \Sigma$ with the metric induced by the lifted BSF structure $\Tilde{q}$ is hyperbolic.
This means that the leaf spaces of the horizontal and vertical foliations together with transverse measures are real trees $T_h$, $T_v$, and we denote the corresponding projections by
\begin{align*}
    \pi_h\colon \tilde{\Sigma} &\rightarrow T_h & \pi_v\colon \tilde{\Sigma} &\rightarrow T_v.
\end{align*}
This allows us to define the \emph{horizontal width} of a subset $X$ of $\Tilde{\Sigma}$ as the diameter of the projection of $X$ onto the vertical tree $T_v$, and analogously the \emph{vertical height} through the projection onto the horizontal tree $T_h$.
The following definition makes this precise:

\begin{definition}[cf. Definition 3.1 of \cite{BHW2}]\label{def:size}
    Let $q$ be a BSF structure on $N$, and $\tilde{\Sigma}$ the universal covering with the lifted BSF structure $\tilde{q}$.
    Let $\gamma$ be a path in $N$, and $\tilde{\gamma}$ a lift of $\gamma$ to $\tilde{\Sigma}$.
    \begin{itemize}
        \item The \emph{total ($q$-)width} of $\gamma$ is the diameter of the projection $\pi_v(\tilde{\gamma})$.
        \item The \emph{total ($q$-)height} of $\gamma$ is the diameter of the projection $\pi_h(\tilde{\gamma})$.
        \item The \emph{($q$-)size} of $\gamma$ is the diameter of $\tilde{\gamma}$ in the metric defined by $\tilde{q}$.
    \end{itemize}

    If $\gamma$ is a simple closed curve, different choices of the base point will yield different lifts and therefore different sizes, widths and heights.
    We therefore define the size, width and height of a closed curve as the respective maximum over all choices of base points.
\end{definition}

\begin{remark}
    As in Remark 3.2 of \cite{BHW2}, we remark that the notion of size does not depend on the flat metric as long as the set of 1-prongs is the same, but that size is essentially topological a topological notion.
\end{remark}

One would expect that a vertical geodesic intersects a horizontally wide, but vertically small curve within bounded time.
This is made precise in:

\begin{lemma}[Lemma 3.6 in \cite{BHW2}]\label{lem:target}
    Let $q$ be a BSF structure on $S$ with ending vertical foliation.
    Then, for any $B$, $\epsilon$ there is a number $L>0$ with the following property:
    suppose that $\gamma \colon [0,1]\rightarrow S$ is a path so that the total width of $\gamma$ is at least $\epsilon$ and the size of $\gamma$ is at most $B$.
    Then, any vertical geodesic for $q$ of length at least $L$ intersects $\operatorname{im}(\gamma)$.
\end{lemma}

Note that in \cite{BHW2}, this is proven for a $P$-resolving covering (using a more general definition than we are), which is always orientable.
Hence, the proof in the non-orientable setting goes through by the same argument.

\section{Hyperbolicity of the fine curve graphs}
In this section we show the hyperbolicity of the fine non-separating curve graphs $\fncg(N)$ for $N$ the real projective plane $\rp$ with up to one puncture.
We will keep this notation throughout the paper.

The material in this section is covered by the work of Bestvina and Fujiwara \cite{Bestvina-refinement}, Kuno \cite{kuno-hyperbolicity} and Bowden, Hensel and Webb \cite{BHW1}.
Hence, we will only cite their statements without proofs.

\subsection{Curve Graphs}\label{subsec:kuno-argument}

The first milestone is to establish the hyperbolicity of the non-separating curve graphs $\ncg(N\setminus P)$, where $P \subset N$ is a finite (possibly empty) set of points.

This statement can already be found in \cite{Bestvina-refinement}.
In \cite{kuno-hyperbolicity}, a similar statement (see Theorem \ref{thm:unif-hyperbolicity}) is proven, again using a guessing geodesics argument.
Kuno's paper follows the strategy of Rasmussen's paper \cite{rasmussen}, which treats orientable surfaces, and whose strategy appeared before in \cite{unicorns}:
one can construct paths between vertices through bicorn surgeries which satisfy the assumptions of the guessing geodesics theorem (Theorem \ref{thm:bowditch-guessing-geodesics}).

We will rely on the following result:
\begin{theorem}[cf. Theorem 1.1 in \cite{kuno-hyperbolicity}]\label{thm:unif-hyperbolicity}
    There exists $\delta > 0$ such that for any finite set $P\subset N$ the fine non-separating curve graph $\ncg(N\setminus P)$ is $\delta$-hyperbolic.

    In addition, if $\alpha$ and $\beta$ are two non-separating curves in minimal position, the set of all non-separating bicorns formed by $\alpha$ and $\beta$ is Hausdorff close to a geodesic between $\alpha$ and $\beta$ in the non-separating curve graph of $N$ (with uniform constants).
\end{theorem}

\begin{remark}
    In \cite{kuno-hyperbolicity}, the theorem is stated for $\rp$ with at least 3 punctures and the Möbius strip with at least 2 punctures respectively, and for surfaces of higher genus.
    This was chosen so that the respective (non-separating) curve graphs have infinite diameter.
    One can easily check that in the low-genus cases with few punctures, the statement is still true with the exception that the graphs will have finite diameter.

    We also added the second statement about bicorn paths:
    what Kuno proves actually satisfies the version of guessing geodesics in Theorem \ref{thm:bowditch-guessing-geodesics}, so we can see that bicorn paths are quasi-geodesics.
    This is analogous to the comments after Theorem 2.2 in \cite{BHW2} on the paper of Rasmussen \cite{rasmussen}.
\end{remark}

\subsection{Fine Curve Graphs}
Let $N$ be defined as before.
We now want to use Theorem \ref{thm:unif-hyperbolicity} to prove that $\fncg(N)$ is hyperbolic.
This is an application of the techniques of section 3 in \cite{BHW1}.

The main idea is: 
if we have two curves $a,b$ that intersect transversely, we can add punctures away from these curves such that they represent different homotopy classes of curves in a punctured surface which are in minimal position.

We first see the following lemma:
\begin{lemma}[cf. Lemma 3.4 of \cite{BHW1}]\label{lem:approx-by-NC}
    Suppose that $\alpha,\beta \in \fncg(N)$ are transverse, and that $\alpha$ and $\beta$ are in minimal position in $N\setminus P$, where $P$ is finite and disjoint from $\alpha$ and $\beta$.
    Then
    \[ d_{\fncg(N)}(\alpha,\beta) = d_{\ncg(N\setminus P)}([\alpha]_{\ncg(N\setminus P)},[\beta]_{\ncg(N\setminus P)}) .\]
\end{lemma}

The proof of this is split into the proofs of two inequalities.
For each part, there are arguments to be made as to why a geodesic in one of the graphs has a corresponding path in the other graph.
These arguments are unchanged by the fact that we are working with non-orientable surfaces, and so is the proof of this lemma.

To prove hyperbolicity, the main effort will be to ensure we can work in the setting of Lemma \ref{lem:approx-by-NC}.
We are ready to prove:

\begin{theorem}[cf. Theorem 3.7 in \cite{BHW1}]\label{thm:fine-cg-hyperbolic}
    The graph $\fncg(N)$ is hyperbolic.
\end{theorem}

\begin{proof}
    As in \cite{BHW1}, we will prove that $\fncg(N)$ is $(\delta +2)$-hyperbolic, where $\delta$ satisfies Theorem \ref{thm:unif-hyperbolicity}.
    The main step is to show that for arbitrary vertices $\mu,\alpha,\beta,\gamma \in\fncg(N)$, we can find $\alpha', \beta', \gamma'$ with the following properties:\begin{enumerate}
        \item $d_{\fncg(N)}(\alpha,\alpha'), d_{\fncg(N)}(\beta,\beta'), d_{\fncg(N)}(\gamma,\gamma') \leq 1$
        \item the curves $\mu, \alpha', \beta', \gamma'$ are pairwise transverse
        \item for any $\kappa, \lambda \in \{\mu, \alpha,\beta,\gamma\}$ we have
        \[d_{\fncg(N)}(\kappa', \lambda') \leq d_{\fncg(N)}(\kappa, \lambda)\]
        where we set $\mu' = \mu$.
    \end{enumerate}
    The existence of such curves follows by the same argument as in \cite{BHW1}.
    After these have been established, one can find the necessary bounds on the Gromov product using the previous Lemma \ref{lem:approx-by-NC}, as the graphs $\ncg(N\setminus P)$ are uniformly hyperbolic.
\end{proof}

From Theorem \ref{thm:unif-hyperbolicity}, we know that paths constructed from bicorn surgeries are Hausdorff-close to geodesics.
The following lemma improves this, and will allow us to construct boundary points in section \ref{subsec:construct-bdary-points}:

\begin{lemma}[cf. Lemma 4.2 of \cite{BHW2}]\label{lem:nice-bicorn-paths}
Let $P\subset S$ be a finite set.
Suppose that $\alpha$ and $\alpha'$ are two non-separating curves in minimal position on $N\setminus P$.
Then any sequence of bicorn surgeries formed by $\alpha$ and $\alpha'$ define a uniform (unparametrised) quasi-geodesic in $\fncg(N)$.
This means that there is a constant $K>0$ such that:
if $(a_i)$ is a sequence of bicorn surgeries of $\alpha$ towards $\alpha'$, then the sequence admits a reparametrisation $i(n)$, so that $n\mapsto a_{i(n)}$ is a $K$-quasi-geodesic.
\end{lemma}

\section{Stable unboundedness}
In this section, we closely follow sections 4 and 5 of \cite{BHW2}.
Our goal here is to show that the techniques from \cite{BHW2} largely still work, with minor modifications to suit the specific surfaces we are interested in.

\subsection{Boundary points from foliations}\label{subsec:construct-bdary-points}
In this section, we describe how boundary points of the fine curve graph are determined by ending foliations, as well as sequences converging to such boundary points and homeomorphisms fixing them.
This is made precise in Theorem 4.1 of \cite{BHW2}:

\begin{theorem}[cf. Theorem 4.1 in \cite{BHW2}]\label{thm:convergence-criterion}
Suppose $(\mathcal{F}_v,\mathcal{F}_h,d)$ is a BSF structure with ending $\mathcal{F}_v$.
Then there is a point $\tau_{\mathcal{F}_v}$ on the Gromov boundary of $\fncg(N)$ with the following properties: \begin{enumerate}
    \item The boundary point $\tau_{\mathcal{F}_v}$ depends only on the foliation $\mathcal{F}_v$.
    \item A homeomorphism of $N$ fixes $\tau_{\mathcal{F}_v}$ as a point on the Gromov boundary if and only if it preserves the foliation $\mathcal{F}_v$.
    \item A sequence of curves $(\beta_i)$ in $\fncg(N)$ converges to $\tau_{\mathcal{F}_v}$ if and only if 
    the following two conditions are satisfied: \begin{itemize}
        \item The sizes of the $\beta_i$ diverge to infinity.
        \item For any $B$, $\epsilon$ there is a number $I$ so that if $b\subseteq \beta_i$ with $i>I$ is a segment of size at most $B$, then it $\epsilon$-fellow travels a leaf segment of $\mathcal{F}_v$ with respect to the metric $d$.
    \end{itemize}
     
\end{enumerate}
\end{theorem}

The proof is rather long, and we prove the three statements individually.

\paragraph{Boundary points.} 
We start with the existence and dependence on the foliation.
For this, we introduce a family of subgraphs of the fine curve graph:
for $L\subset N$ any proper subset, define $\mathcal{D}(L)$ to be subset of the vertex set of $\fncg(N)$ representing the curves which are disjoint from $L$:
    \[ \mathcal{D}(L) \coloneqq \{\alpha \in \fncg(N) \mid \alpha \cup L = \emptyset \}. \]
We will also denote by $\mathcal{D}(L)$ the full subgraph spanned by the set $\mathcal{D}(L)$.
These subgraphs are $K$-quasi-convex, with $K$ independent of $L$, by Lemma \ref{lem:nice-bicorn-paths}, as any $\mathcal{D}(L)$ is invariant under bicorn surgery.

\begin{proposition}[Proposition 4.4 in \cite{BHW2}]\label{prop:foliations-give-bdary-pts}
Consider an ending foliation $\mathcal{F}$ coming from a BSF structure.
Then, there is a unique point of the Gromov boundary $\tau_\mathcal{F}\in \partial_\infty\fncg(N)$ contained in every $\partial_\infty \mathcal{D}(L)$ where $L$ is a finite union of leaf segments of $\mathcal{F}$.

In fact, if $(L_i)$ is a sequence of leaf segments of unboundedly increasing length, then any sequence $(\alpha_i)$, with $\alpha_i \in \mathcal{D}(L_i)$ for all $i$, is an admissible sequence defining the boundary point $\tau_\mathcal{F}$.
\end{proposition}

Set $\omega$ as a base point in $\fncg(N)$; 
we will keep this notation for a choice of base point throughout the rest of the paper.
Consider an increasing sequence $L_1 \subset L_2 \subset \dots$ of leaf segments of diverging length.

We claim that the distance of $\mathcal{D}(L_i)$ to the base point is unboundedly increasing.
To see this, we will argue by contradiction:
if there was a bound on the distances of $\mathcal{D}(L)$ to $\omega$, then by Lemma \ref{lem:covering-criterion} for any $ \delta \in \mathcal{D}(L) $ there would exist a finite branched covering of $N$ with disjoint elevations of $\delta$ and $\omega$.
But for any finite covering, the minimality and increasing length of the $L_i$ imply that for large $i$ there exists a lift $\hat{L_i}$ of $L_i$ such that $\hat{L_i}\cup \hat{\omega}$ is filling for any lift $\hat{\omega}$ of $\omega$, leading to a contradiction.

Now let $\alpha_i$ be a non-separating curve in $\mathcal{D}(L_i)$.
All $\mathcal{D}(L_i)$ are quasi-convex, so the distance of $\mathcal{D}(L_i)$ to the base point gives us a lower bound on the Gromov product of two curves in $\mathcal{D}(L_i)$ (cf. Lemma 6.2 of \cite{bowditch-notes}).
Therefore we have an admissible sequence $\alpha_i$ which defines a boundary point $\tau$.
The same argument shows that a different choice of $\alpha_i$ would yield an equivalent admissible sequence, which shows that this boundary point is the only one contained in the boundaries of all $\mathcal{D}(L_i)$.
If now $L$ is any finite union of leaf segments, we can take a sequence as before with $\alpha_i\in \mathcal{D}(L\cup L_i)$ to finish the proof.

\paragraph{Approaching a boundary point.}
The next step is to characterise sequences of curves which converge to a boundary point determined by an ending foliation.
In this part, we assume that all curves are pairwise intersecting transversally.

First, we show two technical lemmas:

\begin{lemma}[Lemma 4.7 in \cite{BHW2}]\label{lem:size-bounds-distance}
Suppose we have a BSF structure on $N$, and let $\alpha$ be a non-separating curve.
Then for any $L>0$, there is an $R>0$ so that the following is true:
if $\beta$ is a curve of size at most $L$, then 
\[ d_{\fncg}(\alpha,\beta) \leq R. \]
\end{lemma}

The idea here is that for a configuration without bigons formed by $\alpha$ and $\beta$, we can bound their distance by the number of their intersection points, which in turn can be bounded by their sizes.
If there are bigons, we can remove them first with a bounded number of surgeries which only depends on the size of $\beta$ and the curve $\alpha$.

\begin{lemma}[Lemma 4.8 in \cite{BHW2}]\label{lem:bicorn-existence}
Assume that $\mathcal{F}= \mathcal{F}_v$ is the vertical foliation of a BSF structure that is ending.
Suppose that $b$ is a compact embedded segment which is not vertical.
Then for some $p\in b$, the curve formed by a subarc of $b$ and a first return vertical flow line starting in $p$ is non-separating.
\end{lemma}    

For the proof, consider the isotopy classes of first return maps for different points on $b$.
Collapsing $b$, and taking one representative of each homotopy class of first return lines, we can construct a CW-structure of the surface.
Hence, one of the first return lines gives us a loop which is not null-homologous (with $\ZZ_2$-coefficients), and therefore is non-separating.

Now we are ready for:

\begin{lemma}[Lemma 4.6 (Small Width Lemma) in \cite{BHW2}]\label{lem:small-width}
Consider an ending foliation $\mathcal{F}_v$ coming from a BSF structure.
For any $\epsilon$, $B$ there is a $K$ so that the following holds.
If $\beta$ is a simple closed curve and the Gromov product satisfies $(\beta\cdot \tau_{\mathcal{F}_v})_{\omega}>K$, then every segment of $\beta$ of size $B$ has total horizontal width at most $\epsilon$, and thus $\epsilon$-fellow travels a (possibly singular) leaf of $\mathcal{F}_v$.
\end{lemma}

Here, we show the contrapositive:
if $\beta$ contains a segment $b\subset \beta$ with size bounded by $B$ and horizontal width larger than $\epsilon$, then we use Lemma \ref{lem:target} to find a vertical (first return) geodesic segment $\lambda_0$ with length at most $L=L(B,\epsilon)$, which can be chosen to define a non-separating bicorn $\gamma$ with $\beta$ by Lemma \ref{lem:bicorn-existence}. 
By construction, the size of this bicorn is bounded, so by Lemma \ref{lem:size-bounds-distance} we get the bound $d_{\fncg}(\gamma, \omega) \leq R = R(L+B)$.

We next construct a sequence $C_i$ of curves as unions of increasingly long vertical geodesic segments containing $\lambda_0$ and short horizontal geodesic segments.
These converge to $\tau_{\mathcal{F}_v}$ by Proposition \ref{prop:foliations-give-bdary-pts}.
By Lemma \ref{lem:nice-bicorn-paths}, the vertex $\gamma$ is contained in a $K$-quasi-geodesic between $\beta$ and $C_i$ for all $i$.
Together with the bound on $d_{\fncg}(\gamma, \omega)$ this implies that for all $i$ the Gromov product $(\beta\cdot \tau_{\mathcal{F}_v})_\omega$ is bounded from above, and this bound only depends on $B$, $\epsilon$ and $K$.

We can now prove the third part of Theorem \ref{thm:convergence-criterion}:
\begin{proposition}[Proposition 4.9 in \cite{BHW2}]
Suppose that a sequence of curves $\alpha_i$ converges to the boundary point $\tau_\mathcal{F}$ associated to an ending foliation in the Gromov boundary of $\fncg(N)$.
Then the size of the $\alpha_i$ converges to infinity, and convergent subsequences of segments of $\alpha_i$ limit to subsets of either a leaf of pseudo-leaf in the Hausdorff sense.
More precisely, for any $\epsilon$, $B$ and subintervals $J_i\subset \alpha_i$ of size at most $B$, any accumulation point of the $J_i$ in the Hausdorff topology is a (pseudo-)leaf segment of $\mathcal{F}$.
\end{proposition}

The size of the $\alpha_i$ needs to diverge by Lemma \ref{lem:size-bounds-distance}.
This together with Lemma \ref{lem:small-width} implies that for any $L$, $\epsilon$ and large $i$, any segment of $\alpha_i$ of length $L$ is contained in an $\epsilon$-neighbourhood of a (pseudo-)leaf.
Therefore any accumulation points of any segments of bounded length of the $\alpha_i$ are contained in (pseudo-)leaf segments.

\paragraph{Fixing boundary points.}
Last, we prove the second statement of Theorem \ref{thm:convergence-criterion} by characterising stabilisers.

\begin{proposition}[Proposition 4.10 in \cite{BHW2}]
Suppose that $F\colon N\rightarrow N$ is a homeomorphism which preserves the boundary point $\tau_\mathcal{F}$ defined by a vertical foliation $\mathcal{F}$ of a BSF structure that is ending.
Then $F$ preserves the foliation $\mathcal{F}$.
\end{proposition}

The strategy is the following:
For an arbitrary segment $\sigma$ of a regular leaf of $\mathcal{F}$, we choose a sequence of curves $C_i$ constructed from very long vertical leaf segments that contain $\sigma$ and short horizontal leaf segments.
These $C_i$ then converge to $\tau_\mathcal{F}$ by Proposition \ref{prop:foliations-give-bdary-pts}.
As vertices in $\fncg(N)$, the curves $F(C_i)$ also converge to $\tau_\mathcal{F}$ by assumption.
Then, we use the small width lemma (\ref{lem:small-width}) to find that $F(\sigma)$ is eventually contained in the $\epsilon$-neighbourhood of a leaf of $\mathcal{F}$.

\subsection{Exchanging the foliations}
We continue following the strategy of \cite{BHW2}, but we state a slightly more general version of their results.
The next lemma tells us that if a homeomorphism $H$ exists that sends the foliations of one BSF structure to the foliations of another BSF structure then we have a bound on how far $H$ moves the 1-prongs of the first BSF structure.

\begin{lemma}[cf. Lemma 5.3 of \cite{BHW2}]\label{lem:main-technical-lemma}
Let $q_F = (\mathcal{F}_v,\mathcal{F}_h,d)$ and $q_G = (\mathcal{G}_v,\mathcal{G}_h,e)$ be BSF structures, which have a common set $P$ of 1-prongs.
Assume that all horizontal and vertical foliations are ending.
For any $\epsilon > 0$, there is a $K=K(\epsilon) > 0$ so that the following holds:
if $H$ is a homeomorphism so that
\begin{align*}
    \big( H(\tau_{\mathcal{F}_v}) \cdot \tau_{\mathcal{G}_v} \big)_\omega &> K, & \big( H(\tau_{\mathcal{F}_h}) \cdot \tau_{\mathcal{G}_h} \big)_\omega &> K,
\end{align*}
then for any $p\in P$, there is a $p'\in P$ with $d(H(p),p') < \epsilon$.
\end{lemma}

The strategy is again to prove the contrapositive:
We first construct a neighbourhood $V$ of the 1-prongs of $q_G$ so that if the Hausdorff distance of $H(P)$ and $V$ is at least $\epsilon$, then one of the claimed inequalities must fail.

This neighbourhood $V$ will be a union of so-called \emph{bells}.
For each $p'\in P$, we have two bells $U_i$ and $U'_i$:
each consists of the \emph{base}, which is a $q_G$-vertical (respectively $q_G$-horizontal for $U'_i$) leaf segment of length $\epsilon$, such that its middle point lies on a $q_G$-horizontal (respectively $q_G$-vertical) leaf segment through $q_i$ at distance $\epsilon$ from $q_i$, and the $q_G$-horizontal (respectively $q_G$-vertical) segments starting and ending on the base bounding a disc with the base, together with the $q_G$-horizontal (respectively $q_G$-vertical) leaf segment from $q_i$ to the middle of the base.
We can now show that leaves of $H(\mathcal{F}_v)$ need to leave a bell $U_i$ through its base, as otherwise we have an upper bound on the Gromov product through Lemma \ref{lem:small-width}.
The same argument applies to $H(\mathcal{F}_h)$ and a bell $U'_i$.

This leads to a contradiction to the assumption that $H(P)$ is disjoint from all bells, as that would mean that we can choose one leaf of each foliation so that they intersect in exactly one point, and this prevents both leaves to only leave their respective bell through the base.
Section 5 of \cite{BHW2} contains an illustration of the bells.

\subsection{Applications}
We first fix the following notation for the rest of this article:
let $P\subset N$ a set of three punctures (including the puncture of $N$ that might already be there) and an orientable resolving covering $\Sigma\rightarrow N$.
We denote by $\hat{P}$ the lift of $P$ to $\Sigma$.
We now continue in the same way as \cite{BHW2}, but prove a slightly more general theorem:

\begin{theorem}[cf. Theorem 5.2 of \cite{BHW2}]\label{thm:main-thing}
    Let $F$ and $G$ be a Thurston representatives of pseudo-Anosov mapping classes of $N$ relative to a set of 1-prongs $P$ so that $[F]$ and $[G]$ are not conjugate in $\operatorname{Mcg}(N\setminus P)$.
    Then, we have that $F$ and $G$ are not quasi-equivalent in $\operatorname{Homeo}_0(N)$.
\end{theorem}

\begin{proof}
We denote by $\mathcal{F}_{F,v}$ respectively $\mathcal{F}_{F,h}$ the foliations fixed by $F$, which determine boundary points $\tau_\pm$.
Analogously for $G$ we find $\mathcal{F}_{G,v}$ and $\mathcal{F}_{G,h}$ which determine $\upsilon_\pm$.
As our first step, we construct axes for $F$ and $G$:
let $C_0$ a non-separating curve disjoint from $P$, and set $C_i \coloneqq F^i(C_0)$.
Then, the $C_i$ form a quasi-geodesic axis of quality $L$ in $\fncg(N)$ for $F$.
Similarly, we also get a quasi-geodesic $[C_i]$ in $\ncg(N\setminus P)$ for $[F]$, with quality $L'$.
Repeating this construction for $G$, we get the quasi-geodesics $D_i$ with quality $M$ and $[D_i]$ with quality $M'$.

By a similar argument as before, we prove that the quasi-equivalence of $F$ and $G$ would contradict the fact that $[F]$ and $[G]$ are not quasi-equivalent.

By contradiction, we assume that there exists a sequence $H_k$ of homeomorphisms so that
\[\lim_{k\to\infty} h_k(\tau_\pm) = \upsilon_\pm.\]
As $\fncg(N)$ is hyperbolic, we have:
for any $I$ and large $k$, if $i\geq I'(I) > I$, then \begin{enumerate}
    \item $H_k(C_i)$ is within distance $B$ of $D_{\pm j}$ for some $j>I$ and
    \item $H_k(D_{i})$ is within distance $B$ of $C_{\pm j}$ for some $j>I$.
\end{enumerate}
In this case, the distance $B$ depends only on $L$, $M$ and the hyperbolicity constant of $\fncg(N)$.

Our assumption on $[F]$ and $[G]$ together with Proposition \ref{prop:wpd-downstairs} implies that the quasi-axes $[C_i]$ and $[D_i]$ are not quasi-equivalent, and therefore there exist $I$, $K$ such that for all mapping classes $\varphi\in\operatorname{Mcg}(N\setminus P)$ we have for all $i,j\geq I$:
\begin{align}\label{eq:non-inv-mapping-class} \big(\varphi([C_i])\cdot [D_{\pm j}]\big)_{[\omega]} &\leq K - B, & \big(\varphi([D_{i}])\cdot [C_{\pm j}]\big)_{[\omega]} &\leq K - B.\end{align}
Take $\epsilon>0$ such that the $2\epsilon$-neighbourhood of $P$ is disjoint from $C_i$ and $D_i$ for $\lvert i \rvert \leq I'$.
By Lemma \ref{lem:main-technical-lemma}, for large $k$ we get that $H_k(P)$ has distance at most $\epsilon$ from $P$, and further that $H_k(C_{I'})$ respectively $H_k(D_{I'})$ are within distance $B$ of $D_{\pm i}$ respectively $C_{\pm j}$ for $i,j\geq I$.
We can now modify $H_k$ by a $C^0$-small isotopy supported in the $2\epsilon$-neighbourhood of $P$ to get a homeomorphism $H'$ preserving $P$, but not changing the images of $C_{I'}$ and $C_{-I'}$ by our assumption on $\epsilon$.
But then $H'$ defines a mapping class which violates \eqref{eq:non-inv-mapping-class}, proving that $F$ and $G$ are not quasi-equivalent.
\end{proof}

\begin{corollary}\label{cor:inf-dim-qm}
    $\widetilde{\operatorname{QH}}(\operatorname{Homeo}_0(N))$ is infinite-dimensional.
\end{corollary}

\begin{proof}
    First, we can find hyperbolic elements in $\operatorname{Mcg}(N\setminus P)$ which are independent and not quasi-equivalent by Proposition \ref{prop:wpd-downstairs} and Proposition \ref{prop:useful-wpd}.
    The Thurston representatives of must again be independent, and by Theorem \ref{thm:main-thing} also not quasi-equivalent.
    The result now follows from Theorem \ref{thm:bestvina-fujiwara-main-thm}.
\end{proof}

\begin{corollary}\label{cor:scl-criterion}
    Let $F$ be a Thurston representative of a pseudo-Anosov mapping class of $N$ relative to its set of 1-prongs $P$ so that $[F]$ is not conjugate to its inverse in $\operatorname{Mcg}(N\setminus P)$.
    Then $F$ has positive scl in $\operatorname{Homeo}(N)$.
\end{corollary}

\begin{proof}
    Apply Theorem \ref{thm:main-thing} to $F$ and $F^{-1}$.
    The statement then follows by Proposition 5 of \cite{Bestvina} and Bavard duality (Theorem \ref{thm:bavard-duality}).
\end{proof}

\begin{remark}\label{rem:homeo-to-diff}
This result is also true for $\operatorname{Diff}_0(N)$:
the restriction of a homogeneous quasi-morphism on $\operatorname{Homeo}(N)$ to diffeomorphisms is still a homogeneous quasi-morphism, and unboundedness follows as diffeomorphisms are dense (see e.g. Chapter 2 of \cite{hirsch}) and homogeneous quasi-morphisms on homeomorphism groups are continuous (see the appendix of \cite{BHW1}).
\end{remark}

We can now complete the proof of our main theorem:

\begin{proof}[Proof of Theorem \ref{thm:intro}]
The case of the real projective plane follows from Corollary \ref{cor:inf-dim-qm} and Bavard duality (Theorem \ref{thm:bavard-duality}).

It remains to see that the above results are also true for the Möbius strip.
For a point $p$ in $\rp$, we will consider the following surfaces:
first, define $ N \coloneqq \rp\setminus \{p\}$, and second, define the Möbius strip as $M\coloneqq M\setminus B_\epsilon(p)$ for $\epsilon>0$ small.
Note that the interior of $M$ is homeomorphic to $N$.
This implies that restricting a homeomorphism of $M$ to the interior induces a homomorphism
\[r\colon \operatorname{Homeo}(M) \longrightarrow \operatorname{Homeo}(N).\]
In particular, if $F\in\operatorname{Homeo}(M)$ is a homeomorphism, then $\operatorname{cl}(F)\geq \operatorname{cl}(r(F))$, allowing us to apply Corollary \ref{cor:scl-criterion} to certain homeomorphisms of the Möbius strip.

It remains to see that $\widetilde{\operatorname{QH}}(M)$ is infinite-dimensional.
In the proof of Corollary \ref{cor:inf-dim-qm}, we find hyperbolic elements in $\operatorname{Homeo}_0(N)$ which are independent and not quasi-equivalent.
In \cite{Bestvina}, these are used to construct a so-called \emph{Schottky group}.
We claim that one can choose these elements such that they (and therefore the Schottky group) lie in the closure of the image of the map $r$ above.
In the proof of Corollary \ref{cor:inf-dim-qm}, we relied on the WPD action of $\operatorname{Mcg}(N\setminus P)$ on $\ncg(N\setminus P)$ to find these elements.
For the Möbius strip, we can use the same property of the action of $\operatorname{Mcg}(M\setminus P)$ on $\ncg(M\setminus P)$ (which is the same graph as $\ncg(N\setminus P)$).
If now $\phi$ and $\psi$ are two independent quasi-morphisms of $\widetilde{\operatorname{QH}}(N)$, then $\phi\circ r$ and $\psi \circ r$ are again quasi-morphisms.
We check independence as in the proof of Theorem 1 in \cite{Bestvina}, replacing the $f_i$ sequences converging to those elements, as in the proof of Theorem 7.1 of \cite{BHW2}.
Stable unboundedness of $\operatorname{Homeo}_0(M)$ follows directly from $\widetilde{\operatorname{QH}}(M)$ being infinite-dimensional.
\end{proof}

\bibliographystyle{alpha}
\bibliography{refs}

\end{document}